\documentclass[12pt,twoside]{article}
\usepackage{graphicx}
\usepackage{amssymb,array,arydshln} \usepackage{amsfonts,amsbsy,amsthm,amssymb,amsmath,tikz} 
\allowdisplaybreaks \usepackage{cite}
\usepackage{pst-all} 
\def\bpsp{\begin{pspicture}} \def\epsp{\end{pspicture}}

\newtheorem{theorem}{Theorem}[section] \newtheorem{remark}[theorem]{Remark} \newtheorem{example}[theorem]{Example}
\newtheorem{lemma}[theorem]{Lemma} \newtheorem{corollary}[theorem]{Corollary} \newtheorem{definition}[theorem]{Definition}
\newtheorem{proposition}[theorem]{Proposition}  \newtheorem{note}{Note}
\newtheorem{case}{Case} \newtheorem{conjecture}{Conjecture} \newtheorem{question}{Question} 
\newcommand{\bea}{\begin{eqnarray}} \newcommand{\eea}{\end{eqnarray}} \newcommand{\beq}{\begin{eqnarray*}}
\newcommand{\eeq}{\end{eqnarray*}}

  \def\m4{\mbox{\rm ~(mod $4$)}}

\def \bd{\begin{definition}} \def \ed{\end{definition}} \def \bqu{\begin{question}} \def \equ{\end{question}} \def
\bcc{\begin{conjecture}} \def \ecc{\end{conjecture}} \def \bt{\begin{theorem}} \def \et{\end{theorem}} \def
\bl{\begin{lemma}} \def \el{\end{lemma}} \def \bc{\begin{corollary}} \def \ec{\end{corollary}} \def \be{\begin{equation}}
\def \ee{\end{equation}} \def \ben{\begin{enumerate}} \def \een{\end{enumerate}} \def \ba{\begin{array}} \def
\ea{\end{array}} \def \bp{\begin{proposition}} \def \ep{\end{proposition}} \def \bx{\begin{example}} \def
\ex{\end{example}} \def \br{\begin{remark}} \def \er{\end{remark}} \def \bdsc{\begin{description}} \def
\edsc{\end{description}}     \def
\bn{\begin{case}} \def \en{\end{case}} \def \bnt{\begin{note}} \def \ent{\end{note}} \def\1{1\!\!1} 
     
   \def\mm2{\mbox{\rm ~(mod $2$)}}
\def\m4{\mbox{\rm ~(mod $4$)}}  
   \def\m{\mu} 
 \def\1{\textbf{1}} \def\0{\textbf{0}}   
\begin{document} \title{On the eigenvalues and energy of the $A_{\alpha}$-matrix of graphs  } \author{
Nijara Konch$ ^{a} $, A. Bharali$ ^{b} $, S. Pirzada$ ^{c} $\\
$^{a,b}${\em Department of Mathematics, Dibrugarh University, Assam, India}\\
$^{c}${\em Department of Mathematics, University of Kashmir, Srinagar, Kashmir,
India}\\ $ ^{a} $\texttt{rs\_nijarakonch@dibru.ac.in}; $ ^{b} $\texttt{a.bharali@dibru.ac.in}; \\$^{c}$\texttt{pirzadasd@kashmiruniversity.ac.in} } \date{}

\pagestyle{myheadings} \markboth{Konch, Bharali, Pirzada}{On the eigenvalues and energy of the $A_{\alpha}$-matrix of graphs } \maketitle \vskip 5mm
\begin{abstract}
For a graph $G$, the generalized adjacency matrix $A_\alpha(G)$ is the convex combination of the diagonal matrix $D(G)$ and the adjacency matrix $A(G)$ and is defined as $A_\alpha(G)=\alpha D(G)+(1-\alpha) A(G)$ for $0\leq \alpha \leq 1$. This matrix has been found to be useful in merging the spectral theories of $A(G)$ and the signless Laplacian matrix $Q(G)$ of the graph $G$. The generalized adjacency energy or $A_\alpha$-energy is the mean deviation of the $A_\alpha$-eigenvalues of $G$ and is defined as $E(A_\alpha(G))=\sum_{i=1}^{n}|p_i-\frac{2\alpha m}{n}|$, where $p_i$'s are $A_\alpha$-eigenvalues of $G$. In this paper, we investigate the $A_\alpha$-eigenvalues of a strongly regular graph $G$. We observe that $A_\alpha$-spectral radius $p_1$ satisfies $\delta(G)\leq p_1 \leq \Delta(G)$, where $\delta(G)$ and $\Delta(G)$ are, respectively, the smallest and the largest degrees of $G$. Further, we show that the complete graph is the only graph to have exactly two distinct $A_\alpha$-eigenvalues. We obtain lower and upper bounds of $A_\alpha$-energy in terms of order, size and extremal degrees of $G$. We also discuss the extremal cases of these bounds.\\


\end{abstract}
\section{Introduction}
Let $G=G(V(G), E(G))$ be a simple graph with $n$ vertices and $m$ edges, where $V(G)=\{v_1, v_2, \dots, v_n\}$ and $E(G)=\{ e_1, e_2, \dots, e_m\}$ is the vertex set and the edge set of $G$, respectively. The degree of a vertex $v_i$, denoted by $d(v_i)$ or simply $d_i$, is equal to the number of edges incident to $v_i$. An edge connecting vertices $u$ and $v$ is denoted by $uv$. A graph is regular if every vertex possesses equal degree. For a graph $G$ with order $n$, the adjacency matrix associated with $G$ is an $n\times n$ matrix defined as $A(G)=(a_{ij})$, where $a_{ij}=1$ if vertex $v_i$ is adjacent to vertex $v_j$, otherwise 0. The eigenvalues of the adjacency matrix are the eigenvalues of the graph. The notion of graph energy was first introduced by Gutman  \cite{10} as the sum of absolute values of the eigenvalues of $A(G).$ The diagonal matrix $D(G)$ of a graph $G$ is defined as $D(G)=(d_{ij})$, where $d_{ij}=d(v_i)$ (or simply $d_i$) when $i=j$ and otherwise 0. A graph $G$ is said to be connected if  there is a path between every two vertices of $G$, otherwise $G$ is said to be disconnected.
\par  Over the years, several adjacency like matrices associated with a graph have been defined, for instance, Laplacian matrix $L(G)=D(G)-A(G)$, signless Laplacian matrix $Q(G)=D(G)+A(G)$, distance matrix, distance Laplacian (signless Laplacian) matrix $L^{D}(G)$ ($Q^D(G)$) and a plethora of studies had been reported so far, for its spectrum, energies and on various associated parameters. For more details, one can refer \cite{11,12,13,14}.
\par  Nikiforov \cite{1} introduced the generalized matrix $A_\alpha(G)$ (or simply $A_\alpha$, if there is no confusion) as a convex combination of $D(G)$ and $A(G)$ of $G$ to coalesce the spectral theories of adjacency and signless matrices such that $A_\alpha(G)=\alpha D(G)+(1-\alpha) A(G)$, where $0 \leq \alpha \leq 1$ and put forwarded various aspects as well as open problems. From definition, it is clear that $A_0(G)=A(G)$, $A_1(G)=D(G)$, $A_{\frac{1}{2}}(G)=\frac{1}{2}Q(G)$. 
\par Likewise Gutman et al. \cite{15} introduced the Laplacian energy $LE(G)$ of graphs, Abreu et al. \cite{16} introduced the signless Laplacian energy $QE(G)$ of graphs, Gou et al. \cite{6} defined $A_\alpha$-energy of a graph $G$ of order $n$ and size $m$ as $E(A_\alpha(G))=\sum_{i=1}^n\big|p_i-\frac{2\alpha m}{n} \big|$, where $p_i$'s are $A_\alpha$- eigenvalues of $G$. It is seen that $E(A_0(G))=E(G)$ and $E(A_{\frac{1}{2}}(G))=\frac{1}{2} QE(G)$ which shows that $A_\alpha$- energy merges with $E(G)$ and $QE(G)$. More work on $A_\alpha$- spectra and corresponding energy for which one can refer \cite{4,6,17,18,19}.
\par The rest of the paper is organized as follows. Some notations and preliminaries are given in Section 2. In Section 3, we investigate the $A_\alpha$-eigenvalues of a strongly regular graph $G$. We observe that $A_\alpha$-spectral radius $p_1$ satisfies $\delta(G)\leq p_1 \leq \Delta(G)$, where $\delta(G)$ and $\Delta(G)$ are, or simply $\delta$ and $\Delta$ are, respectively, the smallest and the largest degrees of $G$. Further, we show that the complete graph is the only graph to have exactly two distinct $A_\alpha$-eigenvalues. In Section 4, We obtain lower and upper bounds of $A_\alpha$-energy in terms of order, size and extremal degrees of $G$. We also discuss the extremal cases of these bounds.
\section{Notations and Preliminaries}

 Let $(d_1,d_2,\dots,d_n)$ be the degree sequence of $G$. The Zagreb index of $G$, denoted by $M_1$, is defined as $M_1=\sum_{i=1}^{n} d_i^2.$ Throughout the paper, the minimum and maximum degree of a vertex in $G$ is denoted by $\delta$ and $\Delta$ respectively. The trace of a matrix $A$ is the sum of the diagonal entries of $A$, which is denoted by $tr(A)$. For a graph $G$ with $n$ vertices and $m$ edges, let $p_i$, $i\in \{1,\dots, n\}$ be the eigenvalues of $A_\alpha(G)$. The largest eigenvalue, say $p_1$, is called the $A_\alpha$- spectral radius of $G$. The smallest $A_\alpha$ eigenvalue is denoted by $p_n$.  The matrix $J$ and $I$ represent the matrix with all entries 1 and identity matrix of suitable order, respectively.
 \par Let $G$ be an $r$- regular graph with $n$ vertices. $G$ is said to be strongly regular with parameters $(n,r,a,c)$ if $G$ has $a \,(\geq 0)$ common neighbours between any two adjacent vertices and $c\, (\geq 0)$ common neighbours between any two non-adjacent vertices in $G$.
\par The following results will be frequently used in the present work.
\begin{lemma}\em \cite{1} \label{lemma 1}
The eigenvalues of $A_\alpha(K_n)$ are $ p_1=n-1$ and $p_{k}=\alpha n-1$ for $2\leq k\leq n$.
\end{lemma}
\begin{lemma}\em \cite{4} \label{lemma 2}
Let $G$ be a connected graph of order $n$ and size $m$. Then
\begin{enumerate}
\item[\bf (i)] $\sum_{i=1}^{n} p_i=2 \alpha m$
\item[\bf (ii)] $\sum_{i=1}^{n} p_i^2=\alpha^2 M_1+(1-\alpha)^2 2m$
\item[\bf (iii)] $\sum_{i=1}^{n}\eta_i^2=\alpha^2M_1+(1-\alpha)^2 2m-\dfrac{4\alpha^2 m^2}{n}$,
\end{enumerate}
where $\eta_i=p_i-\frac{2\alpha m}{n}$, $i\in \{1,\dots, n\}$, are auxiliary eigenvalues of $A_\alpha(G)$.
\end{lemma}
\begin{lemma} \em \cite{5}\label{Lemma 3}
Let $n\geq 1$ be an integer and $a_1 \geq a_2 \geq \dots \geq a_n$ be some non-negative real numbers. Then $\sum_{i=1}^{n} a_i (a_1+a_n)\geq \sum_{i=1}^{n}a_i^2+n a_1a_n$ with equality if and only if $a_1=a_2=\dots=a_s$ and $a_{s+1}=\dots=a_n$ for some $s \in \{ 1, \dots,n\}$.
\end{lemma}
\begin{lemma} \em \cite{1} \label{Lemma 4}
Let $G$ be a graph with $n$ vertices and $m$ edges. Then $p_1 \geq \frac{2m}{n}$, where equality holds if and only if $G$ is regular.
\end{lemma}
\begin{lemma}\em \cite{7} \label{lemma 5}
Let $B$ be a non-negative $n\times n$ matrix with $n\geq 3$, $\rho(B)$ be the largest eigenvalue of B, $a=tr(B)$ and $b=tr( B^2)$. Then
\begin{align*}
\rho(B)\leq \dfrac{a}{n}+\sqrt{\dfrac{n-1}{n}\left(b-\dfrac{a^2}{n}\right)}.
\end{align*}
\end{lemma}
\begin{lemma} \em \cite{8} \label{lemma 6}
If $G$ is a simple graph with $n \geq 2$ vertices, $m$ edges and vertex degrees $d_1, d_2, \dots, d_n$, then $\sum_{i=1}^n d_i^2 \leq m \left(\frac{2m}{n-1}+n-2 \right)$.
\end{lemma}
\begin{proposition} \label{proposition1}
For any connected graph, the multiplicity of the largest $A_\alpha$-eigenvalue is 1.
\end{proposition}
\begin{proposition}\em \cite{1} \label{proposition2}
Let $\alpha \in [0,1)$ and let $G$ be any connected graph. Let $x$ be a non-negative eigenvector corresponding to the largest $A_\alpha$-eigenvalue $p_1$. Then $x$ is positive and is unique up to scaling.
\end{proposition}
\begin{proposition}\label{proposition2}\em \cite{1}
For a graph $G$ with $\alpha \geq \frac{1}{2}$, $A_\alpha(G)$ is positive semi definite.
\end{proposition}
\begin{theorem}\em \cite{2} \label{theorem 1.1}
(Interlacing Theorem) If $M$ is a real symmetric $n \times n$ matrix and let $\lambda_1(M)\geq \lambda_2(M) \geq \dots \lambda_n(M)$ denote its eigenvalues in non-increasing order. Suppose $A$ is a real symmetric $n \times n$ matrix and $B$ is a principal sub-matrix of $A$ with order $m \times m$. Then for $i\in \{1,\dots,m\}$, $\lambda_{n-m+i}(A) \leq \lambda_{i}(B) \leq \lambda_{i}(A)$. 
\end{theorem}
\begin{theorem}\cite{3}
Let $G$ be a connected graph with $n$ vertices and $m$ edges. Then
\begin{align*} \label{theorem 2}
 M_1 \geq \dfrac{4 m^2 }{n}+\dfrac{ (\Delta-\delta)^2}{2},
 \end{align*}
where equality holds if and only if $G$ is a regular graph.
\end{theorem}
\begin{lemma}  \em \cite{9}\label{lemma8}
Let $G$ be a connected graph with $n$ vertices and $m$ edges. Then
\begin{align*}
M_1\leq \frac{4m^2}{n}+\frac{n}{4}(\Delta-\delta)^2,
\end{align*}
with equality if and only if $G$ is a regular graph.
\end{lemma}
\begin{lemma}\em \cite{20}
Let $G$ be a strongly regular graph with parameters $(n,r,a,c)$ and let $A$ be the adjacency matrix of $G$. Then
$A^2=rI+aA+c(J-I-A)$.
\end{lemma}
\section{On $A_{\alpha}$-eigenvalues of a graph}

First, we obtain  the $A_\alpha$-eigenvalues of a strongly regular graph.
\begin{theorem}
If $G$ is a strongly regular graph with parameters $(n,r,a,c)$ having $m$ edges, then the $A_\alpha$-eigenvalues of $G$ are
\begin{enumerate}
\item[\bf (i)] $r$ with multiplicity 1.
\item[\bf (ii)] $\dfrac{\{2r\alpha+(1-\alpha)(a-c)\}+\sqrt{d}}{2}$ with multiplicity
$$\dfrac{1}{2}\left[ (n-1)+\dfrac{\{2r\alpha+(1-\alpha)(a-c)\}(n-1)+2(r-2\alpha m)}{\sqrt{d}}\right]$$
\item[\bf (iii)] $\dfrac{\{2r\alpha+(1-\alpha)(a-c)\}-\sqrt{d}}{2}$ with multiplicity
$$\dfrac{1}{2}\left[ (n-1)-\dfrac{\{2r\alpha+(1-\alpha)(a-c)\}(n-1)+2(r-2\alpha m)}{\sqrt{d}}\right]$$
\end{enumerate}
where $d =\{2r\alpha +(1-\alpha)(a-c)\}+4\{(r-c)(1-\alpha)^2-r\alpha (1-\alpha)(a-c)-r^2\alpha^2\}$.
\end{theorem}
\begin{proof} Throughout the proof, we use $A_\alpha$ and $D$ to denote the $A_\alpha(G)$ and $D(G)$ respectively. We have
\begin{align*}
 A_\alpha^2 &=(\alpha D+(1-\alpha)A)^2 = \alpha^2 D^2+(1-\alpha)^2 A^2+ 2\alpha (1-\alpha)DA\\
&=\alpha^2 r^2 I+(1-\alpha)^2 [rI+aA+c(J-I-A)]+2r\alpha(1-\alpha)A\\
&=\left\lbrace\alpha^2 r^2+(1-\alpha)^2 r-c(1-\alpha)^2 \right\rbrace I \\
& +\left\lbrace (1-\alpha)^2a+2r\alpha (1-\alpha)-c(1-\alpha)^2 \right\rbrace \frac{1}{(1-\alpha)}\left\lbrace A_\alpha-\alpha D \right\rbrace+c(1-\alpha)^2J.
\end{align*}
This implies that 
$$A_\alpha^2 -\left\lbrace 2r\alpha+(1-\alpha)(a-c)  \right\rbrace A_\alpha -\left\lbrace (r-c)(1-\alpha)^2-r\alpha (1-\alpha)(a-c)-r^2\alpha^2 \right\rbrace I=c(1-\alpha)^2J.$$
Since $G$ is $r$-regular, so $r$ is an eigenvalue of $A_\alpha$ with eigenvector $\textbf{1}$ (i.e., $n\times 1$ column vector with all entries 1). Let $z$ be another eigenvector of $A_\alpha$ corresponding to eigenvalue $\theta \neq r$ such that $A_\alpha \theta=\theta z$. Then
 \begin{align*}
 & A_\alpha^2 z -\left\lbrace 2r\alpha+(1-\alpha)(a-c)  \right\rbrace A_\alpha z -\left\lbrace (r-c)(1-\alpha)^2-r\alpha (1-\alpha)(a-c)-r^2\alpha^2 \right\rbrace Iz \\ & =c(1-\alpha)^2Jz=0
 \end{align*}
 Therefore, as $z\neq 0$, we have
 \begin{align}
 \theta^2  -\left\lbrace 2r\alpha+(1-\alpha)(a-c)  \right\rbrace \theta -\left\lbrace (r-c)(1-\alpha)^2-r\alpha (1-\alpha)(a-c)-r^2\alpha^2 \right\rbrace =0
 \end{align}
 Thus, the remaining eigenvalues of $A_\alpha$ distinct from $r$ which satisfy the equation
 \begin{align*}
 x^2 -\left\lbrace 2r\alpha+(1-\alpha)(a-c)  \right\rbrace x -\left\lbrace (r-c)(1-\alpha)^2-r\alpha (1-\alpha)(a-c)-r^2\alpha^2 \right\rbrace =0
 \end{align*}
 are  $\dfrac{\{2r\alpha+(1-\alpha)(a-c)\}+\sqrt{d}}{2}=\theta_1$ (say) and $\dfrac{\{2r\alpha+(1-\alpha)(a-c)\}-\sqrt{d}}{2}=\theta_2$ (say), where the discriminant of the quadratic equation is $$d=\{2r\alpha +(1-\alpha)(a-c)\}+4\{(r-c)(1-\alpha)^2-r\alpha (1-\alpha)(a-c)-r^2\alpha^2\}.$$
 \par  Using these given parameters for a strongly regular graph, the multiplicities of $\theta_1$ and $\theta_2$ can be attained. If $m_1$ and $m_2$ are the multiplicities of $\theta_1$ and $\theta_2$, then
 $m_1+m_2+1=n$ and $m_1\theta_1+m_2\theta_2+r=2\alpha m$. 
 Solving these two equations and replacing $\theta_1$ and $\theta_2$, we get
 $$m_1=\dfrac{1}{2}\left[ (n-1)+\dfrac{\{2r\alpha+(1-\alpha)(a-c)\}(n-1)+2(r-2\alpha m)}{\sqrt{d}}\right]$$ and
$$m_2=\dfrac{1}{2}\left[ (n-1)-\dfrac{\{2r\alpha+(1-\alpha)(a-c)\}(n-1)+2(r-2\alpha m)}{\sqrt{d}}\right].$$
 \end{proof}
For example, $A_\alpha$-eigenvalues of the strongly regular graph $(5,2,0,1)$, that is, of the cycle of length $5$, are $2, \dfrac{1}{2} \left(-1+5\alpha-\sqrt{5}\sqrt{1-2\alpha+\alpha^2} \right),\dfrac{1}{2} \left(-1+5\alpha+\sqrt{5}\sqrt{1-2\alpha+\alpha^2} \right)$ with multiplicities $1,2,2$ respectively.
\begin{theorem}
Let $G$ be a graph with degree sequence $\{d_1,\dots,d_n\}$. For $0\leq \alpha <1$, the $A_\alpha$-spectral radius of $G$ satisfies $\delta \leq p_1\leq \Delta$. If $G$ is connected, then equality in either case holds if and only if $G$ is regular.
\end{theorem}
\begin{proof}
Let $x=(x_1, x_2, \dots, x_n)^T$ be a non-negative eigenvector corresponding to $p_1$. Since $(\alpha D+(1-\alpha)A)x=\lambda x$, therefore, for $1\leq k\leq n$, we have
\begin{align*}
p_1 x_k &=\alpha d_k x_k+(1-\alpha)\sum_{k \sim j} x_j \\
& \leq    \alpha \Delta x_k+(1-\alpha)\Delta x_k=\Delta x_k.
\end{align*} Similarly,
\begin{align*}
p_1 x_k &=\alpha d_k x_k+(1-\alpha)\sum_{k \sim j} x_j \\
& \geq   \alpha \delta x_k+(1-\alpha)\delta x_k=\delta x_k.
\end{align*}
That is, $\delta x \leq p_1 x \leq \Delta x$ and therefore, by Proposition \ref{proposition2}, we get $\delta \leq p_1 \leq \Delta$. If $G$ is connected and $r$-regular, then $p_1=r=\delta=\Delta$. Conversely, if $p_1=\delta=\Delta$, then
\begin{align*}
\delta x_k &=\alpha d_k x_k+(1-\alpha)\sum_{k \sim j} x_j=\Delta x_k, \,\, 1\leq k\leq n
\end{align*}
which gives $x_1=x_2=\dots=x_n$ and $d_1=d_2=\dots=d_n$. Therefore, $G$ is a regular graph.
\end{proof}
Let $G$ be a graph with $A_\alpha$-eigenvalues $\beta, \gamma, \tau$ with multiplicities say $l, m, k$ respectively. Throughout our work it is represented using the notation $\{[\beta]^l, [\gamma]^m, [\tau]^k\}$. The set of eigenvalues of $A_\alpha(G)$ called spectrum of $A_\alpha(G)$ or $A_\alpha$-spectrum is denoted by $Spec~(A_\alpha(G))$.
\begin{theorem} \label{theorem 1}
Let $G$ be a connected $r$- regular graph with $n$ vertices and $m$ edges. For $\alpha \in [0,1)$, if $Spec~(A_\alpha(G))$=$\{ r, [r \alpha +(1-\alpha)]^a, [r \alpha -(1-\alpha)]^b \}$ for some non-negative integers $a$ and $b$, then $a=0$, $b=r=n-1$ and $G\cong K_{r+1} $.
\end{theorem}
\begin{proof}
Here $1+a+b=n$ and using Lemma \ref{lemma 2}, we get
\begin{align*}
r+a\{r\alpha+(1-\alpha)\}+b\{r\alpha-(1-\alpha)\}=\alpha n r\\
r^2+a \{r\alpha+(1-\alpha)\}^2+b\{ r\alpha-(1-\alpha)\}^2=\alpha^2 n r^2+(1-\alpha)^2 nr
\end{align*}
and hence after simplifying these equations, we obtain $a=0$ and $b=r=n-1$. Therefore, spectrum of $A_\alpha(G)$ reduces to $\{ n-1, [n\alpha-1]^{n-1}\}$. Hence $G\cong K_{r+1}=K_{\frac{2m}{n}+1}$.
\end{proof}
\begin{theorem}\label{theorem 2}
For $0\leq \alpha<1$, a simple connected graph $G$ with order $n \geq 2$ has exactly two $A_{\alpha}$-eigenvalues if and only if $G$ is isomorphic to complete graph $K_n$.
\end{theorem}
\begin{proof}
If $G \cong K_n$, then by Lemma \ref{lemma 1}, the $A_\alpha$-spectrum of $K_n$ is $\{ [n-1]^1, [\alpha n-1]^{n-1}\}$ and hence the sufficient part automatically holds.\\
Conversely, if possible let $G$ posses two distinct eigenvalues say $\beta$ and $\gamma$ ( where $\beta > \gamma$ without loss of generality), and is not complete. This hypothesis is not at all true for $n=2$ and hence result holds trivially. For $n>2$, connectedness of $G$ guarantees that $Spec~(A_\alpha(G))$=$\{[\beta]^1, [\gamma]^{n-1}\}$. Since $G$ is not complete, it contains at least one pair of distinct vertices $u, v \in V(G)$, where $d(u)\geq d(v)$, but $uv \notin E(G)$. Therefore, the vertices $u$ and $v$ induces a principal submatrix of the form $M=\begin{bmatrix}
\alpha d(u) & 0\\
0 & \alpha d(v)
\end{bmatrix}
$ that will be contained in $A_\alpha(G)$. Therefore, by Interlacing theorem \ref{theorem 1.1}, we have
$\gamma =\lambda_2(M)=\alpha d(v).$ Also, as $G$ is connected, there exist some $u' \in V(G)$ such that $u'v \in E(G)$. Thus we can have another principal submatrix $M'$ of $A_\alpha(G)$ induced by $u'$ and $v$, where $M'=\begin{bmatrix}
\alpha d(v) & 1-\alpha \\
1-\alpha & \alpha d(u')
\end{bmatrix}.$
\par Now, the characteristic polynomial of $M'$ is
\begin{align*}
\lambda^2-\lambda \alpha \{d(v)+d(u')\}+ \alpha^2 d(v)d(u')-(1-\alpha)^2=0
\end{align*}
This implies that
\begin{align*}
\lambda=\dfrac{\alpha (d(v)+d(u')) \pm \sqrt{\alpha^2 \{d(v)+d(u')\}^2-4 \alpha^2d(v)d(u')}}{2}.
\end{align*}
From Interlacing theorem, we get $\lambda_n(A_\alpha(G))\leq \lambda_2 (M')\leq \lambda_2 (A_\alpha(G))$.
Therefore
\begin{align*}
 \alpha d(v)=\dfrac{\alpha (d(v)+d(u')) - \sqrt{\alpha^2 \{d(v)+d(u')\}^2-4 \alpha^2d(v)d(u')}}{2},
\end{align*}
implying that
\begin{align*}
\alpha^2 (d(v)-d(u'))^2+4(1-\alpha)^2=\alpha^2 (d(u')-d(v))^2,
\end{align*}
which is a contradiction unless $\alpha=1$ and hence the proof.
\end{proof}
\section{$A_{\alpha}$-energy of graphs}
We start this section with a modification to Theorem 3.1 in \cite{4}. Let $s_i=|p_i-\dfrac{2\alpha m}{n}|, i\in \{1,\cdots, n\}$ such that $s_1\geq s_2\geq \cdots \geq s_n\geq 0$. Therefore we can write the $A_\alpha$-energy of $G$ as $E(A_\alpha(G))=\sum_{i=1}^{n}s_i$.
\begin{theorem} \textbf{(Some modification of Theorem 3. 1 in \cite{4} )}\label{theorem2.1}
If $G$ is a connected graph of order $n \geq 2$ and edges $m\geq 1$, then for $\alpha \in [0,1)$, we have
\begin{align*}
E(A_\alpha(G))> \sqrt{2\left(\alpha^2M_1+(1-\alpha)^22m-\dfrac{4\alpha^2 m^2}{n}\right)},
\end{align*}
where equality can not be attained for connected graph. 
\end{theorem}
\begin{proof} As $s_i=|p_i-\dfrac{2\alpha m}{n}|, i\in \{1,\cdots, n\}$, we have
 \begin{equation}
E(A_\alpha(G))^2=\left(\sum_{i=1}^{n} s_i \right)^2=\sum_{i=1}^{n} s_i^2+2\sum_{1\leq i\leq j\leq n}s_i s_j \label{eqn3}
\end{equation}
Also,
 \begin{align}
2 \sum_{1\leq i\leq j\leq n} s_i s_j =2\sum_{1\leq i\leq j\leq n} \bigg|p_i-\frac{2\alpha m}{n}\bigg| & \bigg| p_j-\frac{2\alpha m}{n} \bigg|=2\sum_{1\leq i\leq j\leq n} \bigg|\left(p_i-\frac{2\alpha m}{n}\right)\left( p_j-\frac{2\alpha m}{n}\right) \bigg| \nonumber  \\
& \geq  2 \bigg|\sum_{1\leq i\leq j\leq n} \left(p_i-\frac{2\alpha m}{n}\right)\left( p_j-\frac{2\alpha m}{n}\right) \bigg| \label{eqn4}\\
&= 2 \bigg|\sum_{1\leq i\leq j\leq n} p_ip_j -\frac{4\alpha^2 m^2(n-1)}{2n}\bigg| \nonumber\\
&= \bigg| \left( 2\sum_{1\leq i\leq j\leq n} p_ip_j-4\alpha^2 m^2\right) +\frac{4\alpha^2 m^2}{n}\bigg| \nonumber
\end{align}
Since 
\begin{equation}
\sum_{i=1}^n \sum_{j=1}^n p_i p_j=\left( \sum_{i=1}^n p_i\right)^2 \label{eqn1}
 \end{equation}
 and 
 \begin{equation}
 \sum_{i=1}^n \sum_{j=1}^n p_i p_j=\sum_{i=1}^n p_i^2+2 \sum_{1\leq i\leq j\leq n}p_ip_j \label{eqn2}
 \end{equation}
 Therefore, using Equation \ref{eqn1} and \ref{eqn2}, we get 
 \begin{align*}
 2 \sum_{1\leq i\leq j\leq n} s_i s_j \geq \bigg| \sum_{i=1}^n p_i^2-\frac{4\alpha^2 m^2}{n} \bigg|=s_i^2\,\, (\text{using Lemma } \ref{lemma 2}(iii))
 \end{align*}
 and hence from Equation \ref{eqn3}, we have
 \begin{align*}
& E(A_\alpha(G))^2\geq 2 \sum_{i=1}^n s_i^2\\
& \text{i.e. }\, E(A_\alpha(G))\geq \sqrt{2\left(\alpha^2 M_1+(1-\alpha)^2 2m-\frac{4\alpha^2 m^2}{n} \right)}
 \end{align*}
 Now, equality in Equation \ref{eqn4} holds if and only if $\left(p_i -\frac{2\alpha m}{n}\right)\geq 0$ for $i\in \{1,\cdots, n\}$ i.e. $p_i \geq \frac{2\alpha m}{n}$ for all $i$. But in such a case Lemma \ref{lemma 2}(i) will not hold for $A_\alpha$-eigenvalues of $G$. Therefore equality can not be attained by connected graph and hence a strict inequality.
\end{proof}
\begin{theorem}\label{theorem 3}
For $\alpha \in [0,1)$, let $G$ be a graph with order $n \geq 2$ and size $m \geq 1$. Let there exist some $t$, $1\leq t\leq n$ such that $s_1=s_2=\dots =s_t=h \geq s_{t+1}=s_{t+2}=\dots=s_n=k$.  Then 
\begin{align*}
E(A_\alpha(G))\geq 2 \sqrt{(\alpha^2M_1+(1-\alpha)^22m-\dfrac{4\alpha^2 m^2}{n})n}\dfrac{\sqrt{s_1s_n}}{s_1+s_n},
\end{align*}
where equality holds if and only if $G\cong \dfrac{n}{2}K_2$ or $G\cong K_{\frac{2m}{n}+1}$ or any of the following holds:
\begin{enumerate}
\item[\bf(i)] $G$ has distinct eigenvalues $\frac{2m}{n}, \frac{2m}{n}(2\alpha-1), \frac{2m\alpha}{n}+(1-\alpha)$ and $\frac{2m\alpha}{n}-(1-\alpha)$.
\item[\bf(ii)] $G$ has distinct eigenvalues $\frac{2m\alpha}{n}+h, \frac{2m\alpha}{n}-h, \frac{2m\alpha}{n}+k $ and $\frac{2m\alpha}{n}-k$, where $h>\frac{2m}{n}(1-\alpha)$ and $h>k$.
\end{enumerate}
\end{theorem}
\begin{proof}
We have \begin{align}
E(A_\alpha(G))&=\sum_{i=1}^n s_i  \geq \dfrac{\sum_{i=1}^{n}s_i^2+n s_1 s_n}{s_1+s_n} \label{equation1.1}\\
&=\dfrac{\alpha^2M_1+(1-\alpha)^22m-\dfrac{4\alpha^2 m^2}{n}+n s_1s_n} {s_1+s_n} \label{eq1} \\
&\geq \dfrac{2\sqrt{(\alpha^2M_1+(1-\alpha)^22m-\dfrac{4\alpha^2 m^2}{n})n s_1s_n}}{s_1+s_n}\nonumber \\
&= 2\sqrt{(\alpha^2M_1+(1-\alpha)^22m-\dfrac{4\alpha^2 m^2}{n})n} \dfrac{\sqrt{s_1s_n}}{s_1+s_n} \label{equation 1.1.1}
\end{align}
Now, equality in Equation \ref{equation1.1} holds if and only if $s_1=s_2=\dots =s_t=h \geq s_{t+1}=s_{t+2}=\cdots=s_n=k$ for some $t \in \{1,2,\dots,n\}$ by Lemma \ref{Lemma 3}. Again, equality in Equation \ref{equation 1.1.1} holds if and only if
\begin{align}
\alpha^2M_1+(1-\alpha)^22m-\dfrac{4\alpha^2 m^2}{n}=n s_1 s_n. \label{equation1}
\end{align} Therefore, we get $h^2 t+k^2 (n-t)=nhk$, which gives
\begin{align}
t(h+k)=nk \label{equation2}.
\end{align}
Now, there are three possibilities for $h$ and $k$.\\
\textbf{Case 1.} $h=k$. Then $s_i=|p_i-\dfrac{2 \alpha m}{n}|=h$ for all $i$, that is, $p_i=h+\dfrac{2\alpha m}{n}$ or $p_i=-h+\dfrac{2\alpha m}{n}$ for $i\in \{1,\dots,n\}$. Therefore, $Spec~(A_\alpha(G))$=$\{ [h+\dfrac{2\alpha m}{n}]^a, [-h+\dfrac{2\alpha m}{n}]^b\}$, where $a+b=n$. Now, from Lemma \ref{lemma 2}(i), we get $a (h+\frac{2 \alpha m}{n})+b(-h+\frac{2\alpha m}{n})=2 \alpha m$, which implies $h(a-b)=0$ and thus $h=0$ or $h\neq 0$ but $a=b$.
\par If $h=0$, then $Spec~(A_\alpha(G))$=$\{ [\frac{2\alpha m}{n}]^n\}$, which is not possible according to Lemma \ref{Lemma 4}.
\par If $h\neq 0$ and $a=b$, then we get $A_\alpha-$ spectral radius of $G$ to be $p_1=h+\frac{2\alpha m}{n}$.
 \par Therefore for $n=2$, $Spec~(A_\alpha(G))$=$\{ h+\frac{2\alpha m}{n}, -h+\frac{2\alpha m}{n}\}$ and thus $G\cong K_2$ (as $G$ has two distinct eigenvalues).
 \par Again, let $n> 2$. Then Proposition \ref{proposition1} suggests that $G$ is disconnected and let $G'$ be a connected component of $G$.  The $A_\alpha-$ spectral radius of $G'$ is $h+\frac{2\alpha m}{n}$. Therefore, $G$ possesses exactly $\frac{n}{2}$ components, say, $G_1, G_2, \dots, G_{\frac{n}{2}}$ such that spectrum of $G_i$ is $\{ h+\frac{2\alpha m}{n}, -h+\frac{2\alpha m}{n}\}$ for $i\in \{1,\dots,\frac{n}{2}\}$. Thus we conclude that each $G_i$ has exactly two vertices and hence $G_i\cong K_2$. Therefore $G\cong \frac{n}{2}K_2$.\\
 \textbf{Case 2.} $h\neq k$ and without loss of generality, let $h > k$. \\
 Let multiplicity of $p_1$ be $l'$ and there exist some $l\geq l'\geq 1$ such that $A_\alpha$-spectrum of $G$ is $Spec~(A_\alpha(G))$= $\{[ \frac{2m\alpha}{n}+h]^{l'},[ \frac{2m\alpha}{n}-h]^l, [ \frac{2m\alpha}{n}+k]^a, [ \frac{2m\alpha}{n}-k]^b \}$, where $a+b=n-l$. By Lemma \ref{Lemma 4}, we have $p_1\geq \frac{2m}{n}$, which implies that $h\geq \frac{2m}{n}(1-\alpha)$.\\
\textbf{Subcase 2.1} $G$ is a connected graph. Then, $A_\alpha$-spectral radius $p_1$ is simple, that is, $l'=1$. Now, for $l\geq 2$, $p_i=\dfrac{2m\alpha}{n}-h$, $i\in \{1,\dots,n\}$.
  If $h=\frac{2m}{n}(1-\alpha)$, then $p_1=\frac{2m}{n}$ and so by Lemma \ref{Lemma 4}, $G$ is a $\frac{2m}{n}$- regular graph. Using Equation \ref{equation1}, we get $k=1-\alpha.$ Therefore, for $l+a+b=n$, we have
\begin{align*}
 Spec~(A_\alpha(G))=\left\{ \frac{2m}{n}, \left[\frac{2m}{n}(2\alpha-1)\right]^{l-1}, \left[\frac{2m\alpha}{n}+(1-\alpha)\right]^a, \left[\frac{2m\alpha}{n}-(1-\alpha)\right]^b\right\}.
\end{align*}
  Again, for connected $G$, if $l\geq 2$ and $h > \frac{2m}{n}(1-\alpha)$, then $G$ is not regular and $A_\alpha$- spectrum is $Spec~(A_\alpha(G))$=$\{\frac{2m\alpha}{n}+h, [\frac{2m\alpha}{n}-h]^{l-1}, [\frac{2m\alpha}{n}+k]^a, [\frac{2m\alpha}{n}-k]^b  \}$, where $a+b+l=n$. Now, from Equation \ref{equation2},  we get $l(h+k)=nk$ which gives $h=\frac{k(n-l)}{l}$. Therefore, from Lemma \ref{lemma 2} (i), we get
  \begin{align*}
  \left(\frac{2m\alpha}{n}+h\right)+(l-1)\left(\frac{2m\alpha}{n}-h\right)+a\left(\frac{2m\alpha}{n}+k)
  +(n-a-l\right)\left(\frac{2m\alpha}{n}-k\right)=2m\alpha.
\end{align*}
This gives $2h-lh+2ak-nk+kl = 0$, so that $(2-l)(n-l)+2al-nl+l^2=0$, as $k\neq 0$.
This implies that $a=b(l-1).$ Therefore, $Spec~(A_\alpha(G))$=$\{ \frac{2m\alpha}{n}+h, [\frac{2m\alpha}{n}-h]^{l-1}, [\frac{2m\alpha}{n}+k]^{b(l-1)}, [\frac{2m\alpha}{n}-k]^b \}$, where $l(b+1)=n$. \\
 Again, if $l=1$, then $Spec~(A_\alpha(G))$=$\{\frac{2m\alpha}{n}+h, [\frac{2m\alpha}{n}+k]^a, [\frac{2m\alpha}{n}-k]^b  \}$, where $a+b=n-1$. Now, if $h=\frac{2m}{n}(1-\alpha)$, then $G$ is a $\frac{2m}{n}-$ regular graph and thus the $Spec~(A_\alpha(G))$=$\{\frac{2m}{n}, [\frac{2m\alpha}{n}+(1-\alpha)]^a, [\frac{2m\alpha}{n}-(1-\alpha)]^b\}$, where $a+b=n-1$. Therefore, by Theorem \ref{theorem 1}, $G\cong K_{\frac{2m}{n}+1}$.
 \par Also, if $h > \frac{2m}{n}(1-\alpha)$, then Equation \ref{equation2} suggests that $h=(n-1)k$  where $k\neq 0$. Therefore, $Spec~(A_\alpha(G))$=$\{ \frac{2m\alpha}{n}+(n-1)k, [\frac{2m\alpha}{n}+k]^a, [\frac{2m\alpha}{n}-k]^{n-a-1} \}$. With this $A_{\alpha}$-spectrum, using Lemma \ref{lemma 2}(i), we get $a=0$ (as $k\neq 0$) and so the spectrum becomes $Spec~(A_\alpha(G))$=$\{\frac{2m\alpha}{n}+(n-1)k, [\frac{2m\alpha}{n}-k]^{n-1} \}$ which is possible only when $G\cong K_n$. This is a contradiction as $G$ is not regular. Therefore, $h\ngtr \frac{2m}{n}(1-\alpha)$ for $l=1, l'=1$.\\
\textbf{Subcase 2.2.} Let $G$ be disconnected.\\
 If $G$ is a regular graph, then $q_1=\frac{2m}{n}$ which is possible only when $h=\frac{2m}{n}(1-\alpha)$ and so as obtained above, we get $k=1-\alpha$. Therefore, for $l+l'+a+b=n$, the $A_\alpha-$ spectrum is
$Spec~(A_\alpha(G))$=$\{ [\frac{2m}{n}]^{l'}, [\frac{2m}{n}(2\alpha-1)]^{l-l'}, [\frac{2m\alpha}{n}+(1-\alpha)]^{a}, [\frac{2m\alpha}{n}-(1-\alpha)]^{b}\}$.
\par If $G$ is not a regular graph, then it implies that $h> \frac{2m}{n}(1-\alpha)$ and therefore proceeding as in Case 1, we get $\frac{b}{a}=\frac{l'}{l}=r$ (say),. Therefore, the resulting $A_\alpha$-spectrum becomes
\begin{align*}
Spec~(A_\alpha(G))=\left\{\left[\frac{2m\alpha}{n}+h\right]^{lr}, \left[\frac{2m\alpha}{n}-h\right]^{l}, \left[\frac{2m\alpha}{n}+k\right]^{a},  \left[\frac{2m\alpha}{n}-k\right]^{ar} \right\},
\end{align*}
where $(1+r)(l+a)=n$.
\end{proof}
Now, we further simplify the lower bound of $E(A_\alpha(G))$ with some estimation on the parameter $\dfrac{\sqrt{s_1s_n}}{s_1+s_n}$
\begin{corollary}\label{corollary 1}
Let $G$ be a connected graph with order $n \geq 3$ and size $m \geq 2$. If $s_n \geq \dfrac{\sqrt{c}}{2n}$, where $c=m\{\alpha^2n^3+n^2(2-4\alpha-\alpha^2)+n(4\alpha-2-2\alpha^2m)+4\alpha^2 m\},$ then for $\alpha \in [0,1)$,
\begin{align}
E(A_\alpha(G))\geq \dfrac{2\sqrt{2}}{3}\sqrt{\left((1-\alpha)^22m+\dfrac{\alpha^2}{2}(\Delta-\delta)^2\right)n}, \label{equation3}
\end{align}
where equality holds if and only if $G \cong K_3$.
\end{corollary}
 \begin{proof}
 By using Lemmas \ref{lemma 5} and \ref{lemma 6}, we have
 \begin{align*}
 s_1&=|p_1-\dfrac{2\alpha m}{n}|=  p_1-\dfrac{2\alpha m}{n} \leq  \dfrac{tr(A_\alpha)}{n}+\sqrt{\dfrac{n-1}{n}\left\lbrace tr A_\alpha^2-\dfrac{(tr A_\alpha)^2}{n}\right\rbrace} -\dfrac{2\alpha m}{n}\\
 & \leq \sqrt{\dfrac{n-1}{n}\left\lbrace \alpha^2 m \left( \frac{2m}{n-1}+n-2\right)+(1-\alpha)^22m -\dfrac{4 \alpha^2 m^2}{n}\right\rbrace}\\
 &= \dfrac{\sqrt{2 \alpha^2 m^2 n+(n^2-n)\{ mn \alpha^2-2 \alpha^2 m+2m(1-\alpha)^2\}-4 \alpha^2 m^2 (n-1)}}{n}\\
 &=\dfrac{\sqrt{c}}{n},
\end{align*}
where $c=m\{\alpha^2n^3+n^2(2-4\alpha-\alpha^2)+n(4\alpha-2-2\alpha^2m)+4\alpha^2 m\}$. Therefore,
$\frac{\sqrt{c}}{2n}\leq s_n \leq s_1 \leq \frac{\sqrt{c}}{n}$ and thus $\frac{\sqrt{s_1 s_n}}{s_1+s_n}\geq \frac{\sqrt{\frac{\sqrt{c}}{2n} \frac{\sqrt{c}}{n}}}{\frac{\sqrt{c}}{2n}+\frac{\sqrt{c}}{n}}=\frac{\sqrt{2}}{3}$. Now, from Theorem \ref{theorem 3}, we get
\begin{align}
E(A_\alpha(G))& \geq 2 \sqrt{\left(\alpha^2M_1+(1-\alpha)^22m-\dfrac{4\alpha^2 m^2}{n}\right)n}~\dfrac{\sqrt{s_1s_n}}{s_1+s_n} \label{equation4} \\
&\geq \dfrac{2\sqrt{2}}{3} \sqrt{\left(\alpha^2M_1+(1-\alpha)^22m-\dfrac{4\alpha^2 m^2}{n}\right)n} \label{equation5}\\
&\geq \dfrac{2\sqrt{2}}{3} \sqrt{\left(\dfrac{\alpha^2}{2}(\Delta-\delta)^2+(1-\alpha)^2 2m \right)n}, \label{equation6}
\end{align}
where equality in Equation \ref{equation4} holds as shown in the cases of Theorem \ref{theorem 3}, equality in Equation \ref{equation5} holds if and only if $s_1=\frac{\sqrt{c}}{n}$ and $\frac{\sqrt{c}}{2n}$ and equality in Equation \ref{equation6} holds for connected regular graph which leads that equality in Equation \ref{equation3} happens only if $G\cong K_{\frac{2m}{n}+1}=K_n$. Also, for $K_n$, we have $s_1=\frac{\sqrt{c}}{n}=(n-1)(1-\alpha)$ and $s_n=\frac{\sqrt{c}}{2n}=1-\alpha$, which gives $n=3$. Therefore equality in Equation \ref{equation3} is attained only if $G\cong K_3$. Conversely, if $G\cong K_3$, then $Spec~(A_\alpha(G))$=$\{ 2, [3\alpha-1]^2\}$ such that $E(A_\alpha(K_3))=4(1-\alpha)=\frac{2\sqrt{2}}{3} \sqrt{\left(\frac{\alpha^2}{2}(\Delta-\delta)^2+(1-\alpha)^2 2m \right)n}.$
 \end{proof}
 \begin{corollary}\label{corollary2}
 For a connected graph $G$ with order $n \geq 3$ and size $m \geq 2$, if $s_n \geq \dfrac{\sqrt{c}}{n^3}$ for the same $c$ as mentioned in Corollary \ref{corollary 1}, then for $\alpha \in [0,1)$
\begin{align*}
E(A_\alpha(G))> \dfrac{2n}{1+n^2}\sqrt{\left((1-\alpha)^22m+\dfrac{\alpha^2}{2}(\Delta-\delta)^2\right)n}. \label{equation3.1}
\end{align*}
 \end{corollary}
\begin{proof}
This can be attained in a similar way as in Corollary \ref{corollary 1} and equality will be reached if and only if $G\cong K_n$ and $s_n=\frac{\sqrt{c}}{n^3}$ and $\frac{\sqrt{c}}{n}$, that is, $s_n=\frac{\sqrt{c}}{n^3}=1-\alpha$ and $\frac{\sqrt{c}}{n}=(n-1)(1-\alpha)$, which gives $n^2=n-1$, a contradiction. Hence equality can not be obtained in this case.
\end{proof}
According to Proposition \ref{proposition2}, for $\alpha\geq \frac{1}{2},$ the smallest $A_\alpha$-eigenvalue i.e. $p_n\geq 0$. In the above corollary, we have assumed that $s_n\geq \frac{\sqrt{c}}{n^3}$ which gives $s_n \rightarrow 0$ as $n\rightarrow \infty$, that is, $p_n\rightarrow 0$ as $n \rightarrow \infty.$ Therefore, for $\alpha \geq \frac{1}{2},$ and for a given small positive number $\epsilon$ to characterize graphs with $s_n\geq \epsilon$ for all $n \in \mathbb{N}$ will be an interesting problem to investigate. Taking $s_n=0$, we have the following improved bounds.
\begin{theorem} \label{theorem4}
Let $G$ be a connected graph with order $n$ and size $m$. If $s_n=0$, then for $\alpha \in [0,1)$, then
\begin{equation*}
E(A_\alpha(G))\geq \dfrac{\alpha^2 M_1+(1-\alpha)^2 2m-\frac{4\alpha^2 m^2}{n}}{s_1}, \label{equation6.1}
\end{equation*}
with equality if and only if any of the following hold.
\begin{enumerate}
\item[\bf(i)] $G$ has distinct $A_\alpha$-eigenvalues $\frac{n}{2}$, $\frac{n(2\alpha-1)}{2}$ and $\frac{n\alpha}{2}$ i.e., $G\cong K_{\frac{n}{2}, \frac{n}{2}}$.
\item[\bf(ii)] $G$ has distinct $A_\alpha$-eigenvalues $h+\frac{2m\alpha}{n}, -h+\frac{2m\alpha}{n}$ and $\frac{2m\alpha}{n}$ for $h>\frac{2m}{n}(1-\alpha)$.
\end{enumerate}
\end{theorem}
\begin{proof}
From Equation \ref{eq1}, we have
\begin{align}
\sum_{i=1}^n s_i\geq \frac{\alpha^2 M_1+(1-\alpha)^2 2m-\frac{4\alpha^2 m^2}{n}}{s_1}  \, \, \, \, \, (\text{as } \, s_n=0)  \label{equation7}
\end{align}
 with equality if and only if $s_1=s_2 = \dots =s_t$ and $ s_{t+1}=\dots = s_{n}=0$ for some $1 \leq t\leq n$. Suppose that Equation \ref{equation7} is an equality. Then $s_i=|p_i-\frac{2m\alpha}{n}|=h$ for $i\in \{1,\dots,t\}$ and $s_j=|p_j-\frac{2m\alpha}{n}|=0$ for $j\in \{t+1, \dots,n\}$, that is, $p_i \in \{ h+\frac{2m\alpha}{n}, -h+\frac{2m\alpha}{n}\}$ and $p_j=\frac{2m\alpha}{n}$.
 \par Since $G$ is connected, we have $p_1=h+\frac{2m\alpha}{n}$, that is, the $A_\alpha$-spectral radius of $G$ is simple and thus $Spec~(A_\alpha(G))$=$\{ h+\frac{2m\alpha}{n}, [-h+\frac{2m\alpha}{n}]^{t-1}, [\frac{2m\alpha}{n}]^{n-t}\}$. \\
 \textbf{Case 1.} If $G$ is $r=\frac{2m}{n}$ regular graph, then $h=\frac{2m}{n}(1-\alpha)$ and thus  $Spec~(A_\alpha(G))$=$\{r ,[r(2\alpha-1)]^{t-1}, [r\alpha]^{n-t}\}$. Therefore, using Lemma \ref{lemma 2}[i] and Lemma \ref{lemma 2}[ii], we have $$r+(t-1) r (2\alpha-1)+(n-t)r\alpha=2m\alpha$$  which gives $t=2$ and the equation $$r^2+r^2 (2\alpha-1)^2+(n-t)r^2\alpha^2=\alpha^2 M_1+(1-\alpha)^2 2m$$ gives $r=\frac{n}{2}$. Therefore, the $A_\alpha$-spectrum of $G$ is $\{\frac{n}{2}, \frac{n(2\alpha-1)}{2}, [\frac{n\alpha}{2}]^{n-2}\}\cong$ $Spec~(A_\alpha(K_{\frac{n}{2}, \frac{n}{2}}))$.\\
 \textbf{Case 2.} If $G$ is non-regular, then $h> \frac{2m (1-\alpha)}{n}$ and therefore $\alpha-$ eigenvalues of $G$ are $h+\frac{2m\alpha}{n}$, $-h+\frac{2m\alpha}{n},$ and  $\frac{2m\alpha}{n}$ with multiplicities $1$, $t-1$ and $n-t$ respectively. For $t=1$, $Spec~(A_\alpha(G))$=$\{h+\frac{2m\alpha}{n}, [\frac{2m\alpha}{n}]^{n-1} \}$ which implies that $G$ has two distinct eigenvalues and hence $G\cong K_n$ which is a contradiction, since $G$ is non-regular. Therefore $t\geq 2$ and so
 \begin{align*}
 \left(h+\frac{2m\alpha}{n}\right)+(t-1)\left(-h+\frac{2m\alpha}{n}\right)+(n-t)\left(\frac{2m\alpha}{n}\right)=2m\alpha,
 \end{align*}
 which gives $t=2$ (as $h\neq 0$). Thus,
 $$ Spec~(A_\alpha(G))=\left\{ h+\frac{2m\alpha}{n}, -h+\frac{2m\alpha}{n}, \left[\frac{2m\alpha}{n}\right]^{n-2}\right\}$$.
\end{proof}
\begin{corollary}
For any connected graph $G$ with order $n$ and size $m$ and $\alpha \in [0,1)$, if $s_n=0$, then  
\begin{equation*}
E(A_\alpha(G))\begin{cases}
> \frac{(1-\alpha)^22m+\frac{(\Delta-\delta)^2}{2}}{\Delta-\frac{2\alpha m}{n}} & \,\, \text{for}\,\,  \text{non regular }\,\,G \\
\geq (1-\alpha)n & \text{for regular G and equality attains iff}\,\, G\cong K_{\frac{n}{2}, \frac{n}{2}}
\end{cases}
\end{equation*}
\begin{proof}
As $s_1=|p_1-\frac{2\alpha m}{n}|=p_1-\frac{2\alpha m}{n}\leq \Delta-\frac{2\alpha m}{n}$, therefore, from Equation \ref{equation6}, we have
\begin{align*}
E(A_\alpha(G))& \geq \dfrac{\alpha^2 M_1+(1-\alpha)^2 2m-\frac{4\alpha^2 m^2}{n}}{s_1} \\
& \geq  \dfrac{(1-\alpha)^2 2m+\frac{\alpha^2 (\Delta-\delta)^2}{2}}{\Delta-\frac{2\alpha m}{n}},
\end{align*}
where first equality is attained if and only if $Spec~(A_\alpha(G))=\{\frac{n}{2}, \frac{n(2\alpha-1)}{2}, [\frac{n\alpha}{2}]^{n-2}\}$ or $Spec~(A_\alpha(G))=\{ h+\frac{2m\alpha}{n}, -h+\frac{2m\alpha}{n}, [\frac{2m\alpha}{n}]^{n-2}\}$ for $h> \frac{2m(1-\alpha)}{n}$ and second equality is attained if and only if $G$ is a regular graph. Therefore, if $G$ is not regular, then $ E(A_\alpha(G))> \frac{(1-\alpha)^22m+\frac{\alpha^2(\Delta-\delta)^2}{2}}{\Delta-\frac{2\alpha m}{n}}$ and if $G$ is regular, then $E(A_\alpha(G))=\frac{(1-\alpha)^22m+\frac{\alpha^2(\Delta-\delta)^2}{2}}{\Delta-\frac{2\alpha m}{n}}=n(1-\alpha)$.
\end{proof}
\end{corollary}
\begin{corollary}
For a connected $r$-regular graph $G$ with order $n$ and size $m$, we have
\begin{align}
E(A_\alpha(G))\geq
\begin{cases}
(1-\alpha)n,\,\,\text{if}\,\, s_n=0\,\, \text{and equality holds if and only if}\,\, G\cong K_{\frac{n}{2},\frac{n}{2}}\\
2(1-\alpha)nr\dfrac{\sqrt{s_n}}{r+s_n}, \,\, \text{if} \,\, s_n>0. \label{equation8}
\end{cases}
\end{align}
\end{corollary}
\begin{proof}
For a connected $r$- regular graph, we have $2m=nr$, $M_1=nr^2$ and the largest $A_\alpha$-eigenvalue is $r$, that is, $s_1=r(1-\alpha)$. Therefore, the first inequality in Equation \ref{equation8} is obvious from Theorem \ref{theorem4}. The second inequality in Equation \ref{equation8} is also obvious from Theorem \ref{theorem 3} for $s_n>0$.
\end{proof}
\begin{theorem}\label{theorem5}
Let $G$ be a graph with order $n$ and size $m$. Then, for $\alpha \in [0,1)$, we have
\begin{align*}
E(A_\alpha(G))\leq \sqrt{\frac{y}{n}}+\sqrt{(n-1)\left( y -\frac{y}{n}\right)},
\end{align*}
where $y=\alpha^2M_1+(1-\alpha)^2 2m-\frac{4\alpha^2m^2}{n} $ and equality holds if and only if $G\cong K_2$ or $G\cong\frac{n}{2}K_2.$
\end{theorem}
\begin{proof}
Let $\eta_i=p_i-\frac{2\alpha m}{n}$ for $i\in \{1,2,\dots,n\}$, where $\eta_1\geq \eta_2\geq \dots\geq \eta_n$. For $\alpha \in [0,1)$, we have $p_1\geq \frac{2m}{n}>\frac{2\alpha m}{n}$ and thus by using Cauchy-Schwarz inequality, we get
\begin{align}
E(A_\alpha(G))&=\big|\eta_1 \big|+\sum_{i=2}^{n}\big| \eta_i\big|\nonumber\\
& \leq  \eta_1+\sqrt{(n-1)\sum_{i=2}^{n} \eta_i^2}\nonumber\\
&=\eta_1+\sqrt{(n-1)\left( \alpha^2M_1+(1-\alpha)^2 2m-\frac{4\alpha^2m^2}{n}-\eta_1^2\right)}\,\, (\text{using Lemma \ref{lemma 2}(iii)})\nonumber
\end{align}
Now, consider the function $f(x)=x+\sqrt{(n-1)\left( \alpha^2M_1+(1-\alpha)^2 2m-\frac{4\alpha^2m^2}{n}-x^2\right)}$ on the defined domain $0< x\leq \sqrt{ \alpha^2M_1+(1-\alpha)^2 2m-\frac{4\alpha^2m^2}{n}} $ where the variable $x$ stands for $\eta_1$. Then
$$f'(x)=1-\dfrac{x\sqrt{n-1}}{\sqrt{ \alpha^2M_1+(1-\alpha)^2 2m-\frac{4\alpha^2m^2}{n}-x^2}}$$ and so $f$ is decreasing on 
$$U=\Bigg\{ x\bigg| \sqrt{\dfrac{ \alpha^2M_1+(1-\alpha)^2 2m-\frac{4\alpha^2m^2}{n}}{n}}\leq x\leq \sqrt{ \alpha^2M_1+(1-\alpha)^2 2m-\frac{4\alpha^2m^2}{n}}\Bigg\}$$ 
and is increasing on 
$$V=\Bigg\{ x\bigg|0<x\leq \sqrt{\dfrac{ \alpha^2M_1+(1-\alpha)^2 2m-\frac{4\alpha^2m^2}{n}}{n}} \Bigg\}.$$ 
Let $y= \alpha^2M_1+(1-\alpha)^2 2m-\frac{4\alpha^2m^2}{n}$. Now, if $\eta_1 \in V$, we have
\begin{align}
E(A_\alpha(G))&\leq f(\eta_1) \leq  f(\sqrt{\dfrac{y}{n}}) \label{equation9}\\
&=f(\sqrt{\dfrac{y}{n}}+\sqrt{(n-1)(y-\dfrac{y}{n})})\label{equation10}
\end{align}
Therefore, the first part of Equation \ref{equation9} is an equality if and only if $\eta_2=\eta_3=\dots=\eta_n$  and the second part of Equation \ref{equation9} is an equality if and only if $\eta_1=\sqrt{\dfrac{ \alpha^2M_1+(1-\alpha)^2 2m-\frac{4\alpha^2m^2}{n}}{n}}$. Finally, Equation \ref{equation9} is an equality if and only if $\eta_1=\sqrt{\dfrac{ \alpha^2M_1+(1-\alpha)^2 2m-\frac{4\alpha^2m^2}{n}}{n}}$  and $\eta_i=p_i-\dfrac{2\alpha m}{n}=\pm\sqrt{\dfrac{ \alpha^2M_1+(1-\alpha)^2 2m-\frac{4\alpha^2m^2}{n}}{n}}$ (using Lemma \ref{lemma 2}(iii)) for $i\in\{2,3,\dots,n\}.$\\
 that is, $p_1=\sqrt{\dfrac{ \alpha^2M_1+(1-\alpha)^2 2m-\frac{4\alpha^2m^2}{n}}{n}}+\dfrac{2\alpha m}{n}$ and  $p_i \in \Bigg\{ \sqrt{\dfrac{ \alpha^2M_1+(1-\alpha)^2 2m-\frac{4\alpha^2m^2}{n}}{n}}+\dfrac{2\alpha m}{n} , -\sqrt{\dfrac{ \alpha^2M_1+(1-\alpha)^2 2m-\frac{4\alpha^2m^2}{n}}{n}}+\dfrac{2\alpha m}{n}\Bigg\}$ for $i\in \{2,3,\dots,n\}$.\\
 \textbf{Case 1.} $G$ is connected, that is, the spectral radius of $G$ has multiplicity 1. So $Spec~(G)=\Bigg\{ \sqrt{\frac{ \alpha^2M_1+(1-\alpha)^2 2m-\frac{4\alpha^2m^2}{n}}{n}}+\frac{2\alpha m}{n} ,\Bigr[ -\sqrt{\frac{ \alpha^2M_1+(1-\alpha)^2 2m-\frac{4\alpha^2m^2}{n}}{n}}+\frac{2\alpha m}{n}\Bigr]^{n-1}\Bigg\}.$\\
  Then, we have
 \begin{align*}
 & \Biggl( \sqrt{\frac{ \alpha^2M_1+(1-\alpha)^2 2m-\frac{4\alpha^2m^2}{n}}{n}}+\frac{2\alpha m}{n}\Biggr)+(n-1)\Biggl(-\sqrt{\frac{ \alpha^2M_1+(1-\alpha)^2 2m-\frac{4\alpha^2m^2}{n}}{n}}+\frac{2\alpha m}{n} \Biggr)\\
& =\sqrt{\frac{ \alpha^2M_1+(1-\alpha)^2 2m-\frac{4\alpha^2m^2}{n}}{n}}(2-n)+2\alpha m,
 \end{align*}
 which will satisfy Lemma \ref{lemma 2}(i) provided $n=2$. Therefore, for a connected graph, equality in Equation \ref{equation10} happens if and only if $G\cong K_2$. \\
 \textbf{Case 2.} If $G$ is disconnected, then there exists some natural number $l$, where $2\leq l<n$ such that  $A_\alpha$- spectrum of $G$ is  $\sqrt{\frac{ \alpha^2M_1+(1-\alpha)^2 2m-\frac{4\alpha^2m^2}{n}}{n}}+\frac{2\alpha m}{n}, -\sqrt{\frac{ \alpha^2M_1+(1-\alpha)^2 2m-\frac{4\alpha^2m^2}{n}}{n}}+\frac{2\alpha m}{n}$ with multiplicity $l$ and $n-l$ respectively. Then we have
 \begin{align*}
& l\Biggl( \sqrt{\frac{ \alpha^2M_1+(1-\alpha)^2 2m-\frac{4\alpha^2m^2}{n}}{n}}+\frac{2\alpha m}{n}\Biggr)+(n-l)\Biggl(-\sqrt{\frac{ \alpha^2M_1+(1-\alpha)^2 2m-\frac{4\alpha^2m^2}{n}}{n}}+\frac{2\alpha m}{n} \Biggr)\\
& =\sqrt{\frac{ \alpha^2M_1+(1-\alpha)^2 2m-\frac{4\alpha^2m^2}{n}}{n}}(2l-n)+2\alpha m,
 \end{align*}
   which will satisfy Lemma \ref{lemma 2}(i) provided $2l+n=0$, that is, $l=\frac{n}{2}$, which implies that $n$ is even natural such that $n\geq 2$. Since $G$ is disconnected and so it is obvious that $G\cong \frac{n}{2}K_2$.
\par Again, if $\eta_1 \in U$, that is
   \begin{align*}
   \sqrt{\frac{\alpha^2M_1+(1-\alpha)^22m-\frac{4\alpha^2m^2}{n}}{n}}\leq \eta_1 \leq \sqrt{\alpha^2M_1+(1-\alpha)^22m-\frac{4\alpha^2m^2}{n}}
   \end{align*}
 and hence same as Equation \ref{equation10}. Therefore the conclusion follows as above.
\end{proof}
\begin{corollary}
Let $G$ be a connected graph with $n$ vertices and $m$ edges. Then
\begin{align*}
E(A_\alpha(G))\leq n\sqrt{\frac{2m(1-\alpha)^2}{n}+\frac{\alpha^2}{4}(\Delta-\delta)^2},
\end{align*}
where equality holds if and only if $G\cong K_2.$
\end{corollary}
\begin{proof}
Here the inequality can be obtained by using Lemma \ref{lemma8}. Also equality holds if and only if $G$ is regular and $n=2$ which implies $G\cong K_2.$
\end{proof}
\begin{corollary}
Let $G$ be a regular graph with order $n$ and size $m$. Then
\begin{align*}
E(A_\alpha(G))\leq n(1-\alpha) \sqrt{\frac{2m}{n}}
\end{align*}
with equality if and only if $G\cong K_2$.
\end{corollary}
\begin{proof}
This can be attained from Theorem \ref{theorem5} considering $G$ to be regular.
\end{proof}
\section{Conclusion}
In this present work, we have obtained $A_\alpha$-spectra of strongly regular graphs with given parameters and some results on $A_\alpha$-eigenvalues of a graph. We have given a modification of a lower bound of $A_\alpha$-energy of a graph $G$, obtained in \cite{4}. Further, we have obtained some bounds of $A_\alpha$-energy of graphs in terms of given  order, size, Zagreb index and extremal degrees of $G$. Equality cases are also analyzed and extremal graphs are obtained.\\

\noindent{\bf Acknowledgments:} The research of S. Pirzada is supported by SERB-DST, New Delhi under the research project number CRG/2020/000109. A. Bharali is thankful to SERB-DST, New Delhi for supporting his research under the Teachers Associateship for Research Excellence (File No.: TAR/2021/000045).


\end{document}